\documentclass[11pt]{amsart}
\usepackage{amsmath}
\usepackage{amssymb}
\usepackage{amscd}
\usepackage{amsthm}

\newtheorem{theorem}{Theorem}[section]
\newtheorem{proposition}[theorem]{Proposition}
\newtheorem{lemma}[theorem]{Lemma}

\newtheorem{corollary}[theorem]{Corollary}

\newtheorem{D}[theorem]{Definition}
\newenvironment{definition}{\begin{D} \rm }{\end{D}}
\newtheorem{R}[theorem]{Remark}
\newenvironment{remark}{\begin{R}\rm }{\end{R}}
\newtheorem{E}[theorem]{Example}

\def\Zee{\mathbb{Z}}
\def\Q{\mathbb{Q}}
\def\Ar{\mathbb{R}}
\def\Cee{\mathbb{C}}
\def\Pee{\mathbb{P}}
\def\Id{\operatorname{Id}}
\def\Ker{\operatorname{Ker}}
\def\Coker{\operatorname{Coker}}
\def\Hom{\operatorname{Hom}}
\def\Ext{\operatorname{Ext}}

\def\im{\operatorname{Im}}

\def\scrO{\mathcal{O}}

\title{The $\partial\bar{\partial}$-lemma for general Clemens manifolds}
\dedicatory{Dedicated to Sir Simon Donaldson for his $60^\text{th}$  birthday}
\author{Robert Friedman}
\address{Department of Mathematics\\
Columbia University\\
New York, NY 10027}
\email{rf@math.columbia.edu}

\begin{document}

\date{\today}

\begin{abstract}
We show that the $\partial\bar{\partial}$-lemma holds for the non-K\"ahler compact complex manifolds of dimension three with trivial canonical bundle constructed by Clemens as deformations of Calabi-Yau threefolds contracted along smooth rational curves with normal bundle of type $(-1, -1)$, at least on an open dense set in moduli. The  proof uses the mixed Hodge structure on the singular fibers and an analysis of the variation of the Hodge filtration for the smooth fibers.
\end{abstract}
\maketitle

\section*{Introduction}

Around 1985, Herb Clemens gave a remarkable construction of compact complex manifolds of dimension three and trivial canonical bundle as follows. Let $X$ be a Calabi-Yau threefold, for example a quintic threefold in $\Pee^4$, and let $C_1, \dots, C_r$ be disjoint smooth rational curves in $X$ such that the normal bundle $N_{C_i/X} \cong \scrO_{\Pee^1}(-1) \oplus \scrO_{\Pee^1}(-1)$ for all $i$,  and such that the classes $[C_1], \dots , [C_r]$ satisfy a linear relation $\sum_im_i[C_i] =0$ in $H^4(X;\Cee)$ with all $m_i\neq 0$ and span $H^4(X;\Cee)$. If $\overline{X}$ is the singular compact complex threefold obtained by contracting the $C_i$ to ordinary double points, then $\overline{X}$ is smoothable, and small smoothings of $\overline{X}$ are compact complex manifolds of dimension three with second Betti number $b_2=0$ and trivial canonical bundle. We will  call any complex manifold obtained in this way a \textsl{Clemens manifold}. If for example $X$ is simply connected and the classes $[C_1], \dots , [C_r]$ generate $H^4(X;\Zee)$, then small smoothings of $\overline{X}$ are diffeomorphic to a connected sum of copies of $S^3\times S^3$. Moreover,  the number $r$ of  curves $C_i$ required in the construction can be arbitrarily large, giving examples of an infinite number of topologically different families of Clemens manifolds. Details of Clemens' construction were given in \cite{Friedman1986}, and the construction was subsequently generalized by Tian \cite{Tian}, Kawamata \cite{Kawamata}, and Ran \cite{Ran}, to the case where the classes $[C_i]$ do not necessarily span $H^4(X;\Cee)$.

Given the very simple topological nature of  Clemens manifolds, it is tempting to speculate that they play a fundamental role in describing the moduli of Calabi-Yau threefolds, see for example Reid \cite{Reid}. It is also natural to ask if their  cohomology in dimension three carries a polarized weight three Hodge structure. While it is easy to see that the Hodge-de Rham spectral sequence degenerates  at $E_1$ (and we recall this argument in the proof of Theorem~\ref{MHS2} below), it is not obvious that the resulting filtrations $F{}^\bullet$ and $\overline{F}^\bullet$ on $H^k$ are $k$-opposed, or equivalently that the $\partial\bar{\partial}$-lemma holds (despite the careless statement on p.\ 107 of \cite{Friedman1991}).  The goal of this paper is to show that indeed the $\partial\bar{\partial}$-lemma holds for a general Clemens manifold. Here general roughly means that the $\partial\bar{\partial}$-lemma holds outside of a proper real analytic subvariety, although it seems likely that in fact it holds for all small smoothings of $\overline{X}$. Unfortunately, the variational methods of this paper do not seem well suited to deciding if the resulting weight three Hodge structures are \textit{polarized}.  Of course, it is a general fact that on a compact complex threefold, if $\omega \in H^0(\Omega^3)$ is nonzero, then $\sqrt{-1}\langle \omega , \bar{\omega}\rangle >0$, where $\langle \cdot, \cdot \rangle$ is the usual pairing on $H^3$. But the remaining Hodge-Riemann inequality for Clemens manifolds, that the Hermitian form on $H^{2,1}$ defined by $\sqrt{-1}\langle \eta , \bar{\eta}\rangle$ is negative definite, seems more difficult to establish. 

One can also ask if there are good metrics on Clemens manifolds whose existence would imply the existence of a Hodge decomposition, and, even better, the Hodge-Riemann inequalities. Results of Fu-Li-Yau \cite{FuLiYau} show the existence of \textsl{balanced metrics} on Clemens manifolds. These are metrics such that the square of the  associated K\"ahler form is $d$-closed (in the case of complex dimension three). However, in general the existence of a balanced metric is not sufficient to imply that the $\partial\bar{\partial}$-lemma holds. 

Fine and Panov \cite{FinePanov} have constructed a complex structure with trivial canonical bundle on $2(S^3\times S^3)\#(S^2\times S^4)$. If $X$ is the corresponding compact complex threefold, they show in addition that there is a nontrivial holomorphic vector field on $X$, i.e.\ that $H^0(X; T_X) \neq 0$. It follows easily that the $\partial\bar{\partial}$-lemma does not hold for $X$. It would be interesting to know if the deformations of $X$ are obstructed, or if there are small deformations of $X$ to a complex manifold for which the $\partial\bar{\partial}$-lemma holds.   

This paper is organized as follows. In Section 1, we collect some general results about Hodge structures and the $\partial\bar{\partial}$-lemma. Section 2 deals with the deformation theory of threefolds with  ordinary double points and trivial dualizing sheaf, as well as the limiting mixed Hodge structures associated to  their smoothings. While all of this material is very well-known to specialists, we give the arguments in some detail to emphasize that it is enough to assume only that a resolution of the singular fiber satisfies the $\partial\bar{\partial}$-lemma. We could replace this assumption by the assumption that  a resolution of the singular fiber is K\"ahler, at the cost of increasing the notational complexity of the argument (see Remark~\ref{inductiveremark}, where we argue that it is enough to consider the case where there is just one smoothing direction in the deformation space). The main calculation is given in Section 3, where we use the nature of the monodromy and the variation of the Hodge filtration to show that the $\partial\bar{\partial}$-lemma holds on a nonempty open subset of the deformation space of smoothings of the singular fiber. Section 4 deals with a question on deformations of compact complex manifolds satisfying the $\partial\bar{\partial}$-lemma, which I first learned of from  Yau.

\medskip
\noindent\textit{Acknowledgements.} It is a pleasure to thank Melissa Liu, Valentino Tosatti, and Shing-Tung Yau for very helpful correspondence, and Joel Fine for several interesting discussions. I would also like to thank Junho Won for his careful reading of a preliminary version of this paper and for many helpful comments.

\section{Some preliminary remarks} 

We begin with a definition of the statement that the $\partial\bar{\partial}$-lemma holds for a compact complex manifold $V$ and its link with the existence of a Hodge structure on the cohomology of $V$.

\begin{definition} Let $V$ be a compact complex manifold and let $A^k(V)$ denote the space
of $C^{\infty }$ $k$-forms on $V$. We say that the
\textsl{$\partial \bar{\partial }$-lemma holds for $V$} if, for all
$k$, and all $\eta \in A^k(V)$ such that $\partial\eta = \bar\partial\eta = 0$, the following
property holds: the form $\eta $ is $d$-exact, i.e.\ there exists a form
$\xi $ such that $\eta =d\xi $ $\iff $ there exists an
$\alpha \in A^{k-2}(V)$ such that
$\eta = \partial \bar{\partial }\alpha $.
\end{definition}

We then have the following \cite[(5.21)]{DGMS} (for the direction (i) $\implies$ (ii), see also \cite[(4.3.1)]{DeligneHodgeII}):

\begin{theorem} Let $V$ be a compact complex manifold. Then the following conditions are equivalent:
\begin{itemize}
\item[\rm(i)] The Hodge-de Rham spectral sequence for $V$ degenerates at $E_1$ and,  for all $k$, if $F^\bullet$ is the corresponding filtration on $H^k(V;\Cee)$, then $F^\bullet$ and $\overline{F}^\bullet$ are $k$-opposed, i.e.\ for all $p$, there is an isomorphism
$$F^p\oplus \overline{F}^{k-p+1}\cong H^k(V;\Cee)$$
induced by the natural inclusions.
\item[\rm(ii)] The $\partial\bar{\partial}$-lemma holds for $V$. \qed
\end{itemize}
\end{theorem}

If either of the above conditions hold, then we define
$$H^{p,q}(V)  = F^p\cap \overline{F}^q\subseteq H^{p+q}(V;\Cee).$$
 Equivalently, by \cite[(5.4)(i)]{DeligneLefschetz}, 
$H^{p,q}(V)$ is the set of $\alpha \in H^{p+q}(V;\Cee)$ such that  there exists a de Rham representative for $\alpha$ of type $(p,q)$. Then we have the usual Hodge decomposition
$$H^k(V;\Cee) =\bigoplus_{p+q=k}H^{p,q}(V), \qquad \text{ and } \overline{H ^{p,q}(V)}= H^{q,p}(V).$$

\begin{remark}\label{Remark1.3} (i) The conditions that the Hodge-de Rham spectral sequence for $V$ degenerates at $E_1$ and that,  for all $k$,   the filtrations  $F^\bullet$ and $\overline{F}^\bullet$ are $k$-opposed are both open conditions. Hence the condition that the $\partial\bar{\partial}$-lemma holds for $V$ is an open condition. 

\smallskip 
\noindent (ii) Let $V$ be a compact complex manifold for which the $\partial\bar{\partial}$-lemma holds. It does not seem to be clear if this property is inherited by a closed holomorphic submanifold $Z$. However, if $Z$ is a closed holomorphic submanifold  of $V$ and the $\partial\bar{\partial}$-lemma holds for $Z$, then it is easy to see that the inclusion and the Gysin homomorphism are morphisms of Hodge structures. More generally if $Z$ is a compact complex  manifold such that  the $\partial\bar{\partial}$-lemma holds for $Z$, and $f\colon Z\to V$ is   a holomorphic map, then $f^*$ and $f_*$ are morphisms of Hodge structures (with the appropriate shift in the case of $f_*$ or the Gysin homomorphism). This follows since clearly $f^*H^{p,q}(V) \subseteq H^{p,q}(Z)$ and because the Poincar\'e duality isomorphism $(H^k(V;\Q))^*\cong H^{2n-k}(V;\Q) \otimes \Q(n)$ is an isomorphism of Hodge structures. 

\smallskip 
\noindent (iii) Let $V$ be a compact complex manifold for which the $\partial\bar{\partial}$-lemma holds, and let $Z$ be a submanifold of $V$ for which the $\partial\bar{\partial}$-lemma also holds. If $\rho\colon \widetilde{V} \to V$ is the blowup of $V$ along $Z$, then  the $\partial\bar{\partial}$-lemma holds for $\widetilde{V}$. There have been a number of recent preprints which address this issue \cite{RYY}, \cite{ASTT}, \cite{Stelzig}.  The main point is to show that the Hodge-de Rham spectral sequence degenerates at the $E_1$ page. This can be done by an examination of the Leray spectral sequence $E_2^{p,q} = H^p(V; R^q\rho_*\Omega^k_{\widetilde{V}}) \implies H^{p+q}(\widetilde{V};\Omega^k_{\widetilde{V}})$ and the well-known computation of the Betti numbers of $\widetilde{V}$. In fact, the Leray spectral sequence above  additionally degenerates at the $E_2$ page. The above results then  follow easily from a computation due to Gros \cite[IV Th\'eor\`eme 1.2.1]{Gros}; compare also \cite[Proposition (3.3)]{GNV}. 
\end{remark}

\begin{lemma}\label{prelimlemma1} Let $V$ be a compact complex manifold of dimension $d$ for which the Hodge-de Rham spectral sequence degenerates at $E_1$ and let $F^\bullet$ be the corresponding filtration on $H^d(V;\Cee)$. Then $F^\bullet$ is isotropic for cup product, in the sense that, for all $k$,  $(F^k)^\perp = F^{d-k+1}$. 
\end{lemma}
\begin{proof} First, we claim that $F^{d-k+1}\subseteq (F^k)^\perp$. Every element of $F^k$ has a de Rham representative  $\eta$ with $d\eta =0$ and 
$$\eta = \sum_{\ell\geq k}\eta^{\ell, d-\ell}, \qquad \eta^{\ell, d-\ell} \in A^{\ell, d-\ell}(V),$$
and similarly for elements of  $F^{d-k+1}$. Thus, if $\xi\in F^k$ and $\xi'\in F^{d-k+1}$, then $\xi\smile \xi'$ has a de Rham representative which is a sum of forms of type $(d+a, d-a)$, $a\geq 1$, and hence is $0$, so that  $\xi\smile \xi'=0$.

Since $V$ has dimension $d$, Kodaira-Serre duality implies that 
$$\dim H^i(V;\Omega^{d-i}_V) = \dim H^{d-i}(V;\Omega^i_V).$$
It is then easy to see that $F^{d-k+1}$ and $(F^k)^\perp$ have the same dimension. Since $F^{d-k+1}\subseteq (F^k)^\perp$, we must have $F^{d-k+1}= (F^k)^\perp$.
\end{proof}

\begin{lemma}\label{prelimlemma21}  Let $V$ be a compact complex manifold of dimension $d$ for which the Hodge-de Rham spectral sequence degenerates at $E_1$. Then the natural map
$$F^1 \oplus \overline{F}^d\to H^d(V;\Cee)$$
is an isomorphism, and hence so is the map $F^d \oplus \overline{F}^1\to H^d(V;\Cee)$. 
\end{lemma}
\begin{proof} Since the codimension of $F^1$ in $H^d(V; \Cee)$ is the dimension of $\overline{F}^d$, it suffices to show that $F^1 \cap \overline{F}^d = 0$. Let $\omega$ be a holomorphic $d$-form and suppose that $\bar{\omega}\in F^1$.  By Lemma~\ref{prelimlemma1}, $F^1 = (F^d)^\perp$ and hence $\displaystyle \int_V\omega \wedge \bar{\omega} = 0$. It follows that $\omega =   \bar{\omega} =0$, and hence that $F^1 \cap \overline{F}^d = 0$ as claimed.
\end{proof}

Let $X$ be a compact complex manifold of dimension three for which the Hodge-de Rham spectral sequence degenerates at $E_1$ and let $F^\bullet$ be the corresponding filtration on $H^3(X;\Cee)$. Suppose that  that $H^i(X; \scrO_X) = H^0(X; \Omega_X^i) = 0$ for $i=1,2$. It follows that, for $n$ odd, $n\neq 3$, $H^n(X;\Cee) = 0$, and for $n=2k$ even, the filtration $F^\bullet$ on $H^{2k}(X;\Cee)$ satisfies: $F^pH^{2k}(X;\Cee) =0, p> k$, and $F^pH^{2k}(X;\Cee) =H^{2k}(X;\Cee)$, $p\leq k$. Thus trivially the filtrations $F^\bullet$ and $\overline{F}^\bullet$ are $2k$-opposed and induce  a Hodge structure on $H^{2k}(X;\Cee)$ for which  $H^{2k}(X;\Cee) = H^{k,k}(X)$.

\begin{corollary}\label{prelimcor} Let $X$ be a compact complex manifold of dimension $3$ for which the Hodge-de Rham spectral sequence degenerates at $E_1$ and let $F^\bullet$ be the corresponding filtration on $H^3(X;\Cee)$.  Suppose that  $H^i(X; \scrO_X) = H^0(X; \Omega_X^i) = 0$ for $i=1,2$. Then the $\partial\bar{\partial}$-lemma holds for $X$ $\iff$ $H^3(X;\Cee) \cong F^2\oplus\overline{F}^2$. \qed
\end{corollary}

\section{The limiting mixed Hodge structure}

\subsection{Notation} We fix the following notation for the rest of this section: Let $X_0$ be a compact complex manifold of dimension $3$ with trivial canonical bundle  for which the $\partial\bar{\partial}$-lemma holds. We assume further  that $H^i(X_0; \scrO_{X_0}) =H^0(X_0; \Omega_{X_0}^i) = 0$ for $i=1,2$.  By the Tian-Todorov theorem, the Kuranishi deformation space for $X_0$ can be identified with  the germ of the origin in $H^1(X_0; T_{X_0}) \cong H^1(X_0; \Omega^2_{X_0})$, and is thus a smooth germ of dimension $h$, where 
$$h = h^{2,1}(X_0) = \dim H^1(X_0; \Omega^2_{X_0}).$$

 Let $C_1, \dots, C_r$ be disjoint smooth curves in $X_0$ such that, for all $i$,  $C_i\cong \Pee^1$ and the normal bundle $N_{C_i/X_0} \cong \scrO_{\Pee^1}(-1) \oplus \scrO_{\Pee^1}(-1)$, i.e.\ is of type $(-1, -1)$. The  $C_i$ can be contracted in $X_0$ to points $p_i$, yielding a singular compact threefold $\overline{X}_0$.
We assume that the cohomology classes  $[C_i]$ of the $C_i$ satisfy a linear relation in $H^4(X_0; \Cee)$ of the form
$$\sum_{i=1}^rm_i [C_i] =0, m_i\in \Q,$$
where $m_i \neq 0$ for every $i$, and  that the $[C_i]$ span a subspace of $H^4(X_0; \Cee)$ of dimension $r-1$, so that no $r-1$ of the $[C_i]$ are linearly dependent. Note that we can and shall consider the case $r=1$, in which case the above assumption is simply that $[C_1] = 0$ in $H^4(X_0; \Cee)$.

\subsection{The deformation space}  To analyze the deformation theory of $\overline{X}_0$, let  $\mathbb{T}^i_{\overline{X}_0} =\Ext^i(\Omega^1_{\overline{X}_0}, \scrO_{\overline{X}_0})$ be the objects of Lichtenbaum-Schlessinger theory. Then by \cite[Theorem 4.4]{Friedman1986}, there is an exact sequence
$$0\to H^1(\overline{X}_0; T^0_{\overline{X}_0}) \to \mathbb{T}^1_{\overline{X}_0}  \to \Cee \to 0,$$
where the last term $\Cee$ is identified with the kernel of the fundamental class map
$$\bigoplus_i\Cee[C_i] \to H^4(X_0; \Cee) =H^2(X_0; \Omega^2_{X_0}),$$
$T^0_{\overline{X}_0}$ is the tangent sheaf of $\overline{X}_0$,  and $H^1(\overline{X}_0; T^0_{\overline{X}_0}) \cong H^1(X_0; T_{X_0})$ by \cite[(3.4)]{Friedman1986} and the following remarks.  

 The space $\overline{X}_0$ is smoothable. More precisely, there is the following result due independently to Tian \cite{Tian}, Kawamata \cite{Kawamata}, and Ran \cite{Ran}:

\begin{theorem} The locally semi-universal deformation space for $\overline{X}_0$ can be identified with the germ of the origin in $\mathbb{T}^1_{\overline{X}_0}$, and thus is a smooth germ of dimension $h+1$. Moreover, the germ of the hyperplane $H^1(\overline{X}_0; T^0_{\overline{X}_0})$ corresponds to locally trivial deformations of $\overline{X}_0$, which are identified with deformations of $X_0$. The points lying over the (germ of) $\mathbb{T}^1_{\overline{X}_0} -H^1(\overline{X}_0; T^0_{\overline{X}_0})$ are smooth compact complex manifolds of dimension $3$ with trivial canonical bundle. \qed
\end{theorem}

\begin{remark} Tian proves the theorem under the assumption that the $\partial\bar{\partial}$-lemma holds for $X_0$. Kawamata's result is stated under the hypothesis that $X_0$ is projective, but the proof seems to work in much greater generality. Ran's proof apparently only uses the degeneration of the Hodge-de Rham spectral sequence for $X_0$. If we make the very stringent assumption that the classes $[C_i]$ span  $H^2(X_0; \Omega_{X_0}^2)$ (the main case of interest in this paper), then the above theorem is proved in \cite{Friedman1986}, assuming only that  $K_{X_0}\cong \scrO_{X_0}$.
\end{remark}

\begin{remark}\label{inductiveremark} In what follows, to simplify notation, we will use the fact that it is possible to smooth ``one dimension at a time." More precisely, suppose that the $C_i$ are smooth rational curves of the type considered, but without the assumption that no $r-1$ of the classes $C_i$ are linearly  dependent. We can then reorder the $C_i$ so that no $s-1$ of the classes $[C_1], \dots, [C_s]$ are linearly independent and that $[C_1], \dots, [C_s]$ are linearly dependent. Let $X_1$  be a general smooth threefold  which is a small smoothing of the singular threefold $X_0'$ with  double points obtained by contracting $C_1, \dots, C_s$. Then  $K_{X_1}$ is trivial and (as we shall show) satisfies the $\partial\bar{\partial}$-lemma. The classes $C_{s+1}, \dots, C_r$ then deform to curves in $X_1$, satisfying a linear relation with nonzero coefficients, and we can then repeat the construction.
\end{remark}

\subsection{A normal crossings model} We turn next to semistable models for the deformations of $\overline{X}_0$. Let $\widetilde{X}_0$ be the blowup of $\overline{X}_0$ at the double points, or equivalently of $X_0$ along the curves $C_i$. By Remark~\ref{Remark1.3}(iii), $\widetilde{X}_0$ also satisfies the $\partial\bar{\partial}$-lemma (as is easy to check directly in this special case). Moreover one easily checks that $H^3(\widetilde{X}_0) \cong H^3(X_0)$ is an isomorphism of Hodge structures. The exceptional divisors $Q_i$ over $p_i$, or $C_i$, are smooth quadrics. Thus $Q_i \cong \Pee^1\times \Pee^1$ and the normal bundle $N_{Q_i/\widetilde{X}_0}$ of $Q_i$ in $\widetilde{X}_0$ is $\scrO_{Q_i}(-1, -1)$ (using the standard notation for line bundles on $Q_i$). For each $i$, let $E_i$ be a smooth quadric threefold in $\Pee^4$ and identify $Q_i$ with a smooth hyperplane section of $E_i$, also denoted $Q_i$, by some choice of isomorphism. (Since every element in the neutral component  of the automorphism group of $Q_i$ is induced by restriction from an automorphism of $E_i$, the choice of an isomorphism is irrelevant.) Thus $N_{Q_i/E_i}\cong \scrO_{Q_i}(1, 1)$. Let $\widetilde{Y}_0 = \widetilde{X}_0 \amalg \coprod_i E_i$ and let
$$Y_0 = \widetilde{X}_0 \amalg \coprod_i E_i/\sim\, \, ,$$
where the equivalence relation $\sim$ means that we glue $Q_i \subseteq E_i$ to $Q_i \subseteq \widetilde{X}_0$ by the choice of an isomorphism above. Note that $Y_0$ is in the natural way a $d$-semistable variety with normal crossings in the sense of \cite[(1.13)]{Friedman1983}. Let $\nu\colon \widetilde{Y}_0 \to Y_0$ be the normalization morphism.  We can exhibit a model for the smoothings of $Y_0$ as follows. Let $\bar{\pi}\colon \mathcal{X} \to \overline{S}$ be the germ of the locally semi-universal deformation of $\overline{X}_0$, where we can identify $\overline{S}$ with the germ about the origin in $\mathbb{T}_{\overline{X}_0}$. Let $S\to \overline{S}$ be the double cover of $\overline{S}$ branched along the smooth hypersurface   $\overline{S} \cap H^1(\overline{X}_0; T^0_{\overline{X}_0})$ and let $\overline{\mathcal{Y}} \to S$ be the pulled back family. If $D$ is  the ramification divisor of the cover $S\to \overline{S}$ or equivalently the inverse image of $H^1(\overline{X}_0; T^0_{\overline{X}_0})$ in $S$, then $D$ is the discriminant locus of $\bar{\pi}$, the fibers of $\overline{\mathcal{Y}}$ over $D$ have $r$ ordinary double points corresponding to the singular points and the singularities of the total space $\overline{\mathcal{Y}}$ are locally analytically isomorphic to products of ordinary double points of dimension $4$ with $D$. Blowing up these singular points gives a proper flat morphism $\pi\colon \mathcal{Y} \to S$, where $\mathcal{Y}$ is smooth, the discriminant locus of $\pi$ is $D$, and the fibers of $\pi$ over $D$ are locally trivial deformations of the normal crossings varieties $Y_0$ described above (and hence also have normal crossings). Let $\mathcal{Y}_D=\pi^{-1}(D)$. Thus $\mathcal{Y}_D$ is a divisor with normal crossings in $\mathcal{Y}$. For $s\notin D$, the fiber $Y_s$ of $\pi$ is identified with the corresponding smooth fiber $X_{\bar{s}}$ of $\bar{\pi}$, where $\bar{s}\in \overline{S}$ is the point lying under $s$. 

\subsection{A mixed Hodge structure  on $Y_0$}
By convention, all cohomology is with $\Cee$-coefficients unless otherwise specified.  We have the Mayer-Vietoris sequence for $Y_0$:
$$0 \to \Cee_{Y_0} \to \nu_*\Cee_{\widetilde{Y}_0} \to \bigoplus _i(j_i)_*\Cee_{Q_i} \to 0,$$
where $j_i \colon Q_i \to Y_0$ is the inclusion.   Using the fact that $H^1(Q_i) = H^3(Q_i) = 0$, we get an exact sequence
\begin{gather*}0 \to H^2(Y_0) \to H^2(\widetilde{X}_0)\oplus \bigoplus_iH^2(E_i) \to \bigoplus_iH^2(Q_i)\to \\
\to H^3(Y_0) \to H^3(\widetilde{X}_0) = H^3(X _0) \to 0.
\end{gather*}
If $b=\dim H^2(X_0)$ is the second Betti number $b_2(X_0)$, then the second Betti number $b_2(\widetilde{X}_0)$ of $\widetilde{X}_0$ is $b+r$ and hence  $\dim H^2(\widetilde{X}_0)\oplus \bigoplus_iH^2(E_i)=b+2r$. Moreover, $\bigoplus_iH^2(Q_i)\cong \Cee^{2r}$. In fact, $H^2(Q_i) \cong \Cee[\sigma_i]\oplus \Cee[f_i]$, where $f_i$ is a fiber of the morphism $Q_i \to C_i$, and $\sigma_i$ is a fiber of the ``other ruling" on $Q_i \cong \Pee^1\times \Pee^1$. Then, taking the positive generator $[Q_i]$ of $H^2(E_i)$, the homomorphism 
$H^2(E_i)\to H^2(Q_i)$ sends $[Q_i]$ to $[\sigma_i] + [f_i]$.  The homomorphism 
$H^2(\widetilde{X}_0)\to H^2(Q_i)$ sends $[Q_i]$ to $-[\sigma_i] - [f_i]$ and sends a class of the form $\rho^*\xi$, where $\rho\colon \widetilde{X}_0 \to X_0$ is the blowup morphism, to $(\xi \cdot [C_i])[f_i]$. A brief computation shows the following:

\begin{proposition}\label{prop2.4}  {\rm(i)} Let $W_2$ be the image of $\bigoplus_iH^2(Q_i)$ in $H^3(Y_0) = W_3$. Then $W_2$ has rank one and $W_3/W_2 \cong H^3(X _0)$. 

\smallskip
\noindent {\rm(ii)} $H^1(Y_0) = H^5(Y_0) = 0$. 

\smallskip
\noindent {\rm(iii)} $H^2(Y_0)$ has dimension $b+1$, and is isomorphic to the following subgroup of $H^2(\widetilde{X}_0)\oplus \bigoplus_iH^2(E_i)$
$$\left\{\rho^*\xi + \sum_i a_iq_i' + \sum _ib_iq_i'': a_i = b_i \text{ and } \xi \cdot [C_i] = 0  \text{ for all $i$ }\right\},$$
where $q_i'$ is the class of $Q_i$ in $H^2(\widetilde{X}_0)$ and $q_i''$ is the class of $Q_i$ in
$H^2(E_i)$.

\smallskip
\noindent {\rm(iv)} $H^4(Y_0) \cong  H^4(X_0)\oplus \Cee^r$ has dimension  $b+r$.
\qed
\end{proposition}

Part (i) of Proposition~\ref{prop2.4} gives a weight filtration on $H^3(Y_0)$, defined over $\Q$, with $W_1 =0$. There are also trivial (increasing) filtrations on $H^k(Y_0)$ for $k\neq 3$: take $W_\ell = H^k(Y_0)$ for $\ell\geq k$ and $W_\ell =0$ for $\ell < k$. To construct a Hodge filtration, we can use the complex $\Omega_{Y_0}^{\bullet}/\tau_{Y_0}^{\bullet}$ of \cite[(1.5)]{Friedman1983}, where $\Omega_{Y_0}^1$ is the sheaf of K\"ahler differentials on $Y_0$, $\Omega_{Y_0}^{\bullet} = \bigwedge^{\bullet}\Omega_{Y_0}^1$, and $\tau_{Y_0}^{\bullet}$ is the subcomplex of ``torsion differential," i.e.\ those supported on $(Y_0)_{\text{sing}}$. By \cite[(1.5)]{Friedman1983}, $(\Omega_{Y_0}^{\bullet}/\tau_{Y_0}^{\bullet}, d)$ is a resolution of the constant sheaf $\Cee_{Y_0}$, and there is an exact sequence 
$$0 \to \Omega_{Y_0}^{\bullet}/\tau_{Y_0}^{\bullet} \to \nu_*\Omega_{\widetilde{Y}_0}^{\bullet}\to \bigoplus _i(j_i)_*\Omega_{Q_i}^{\bullet}\to 0.$$
Taking hypercohomology gives the Mayer-Vietoris sequence above. 

\begin{theorem}\label{MHS1} The spectral sequence with $E_1$ page 
$$E_1^{p,q} =H^q(Y_0; \Omega_{Y_0}^p/\tau_{Y_0}^p)\implies \mathbb{H}^{p+q}(Y_0; \Omega_{Y_0}^{\bullet}/\tau_{Y_0}^{\bullet})=H^{p+q}(Y_0) $$
 degenerates at $E_1$. The corresponding filtration $F^\bullet$ on $H^k(Y_0)$, together with the weight filtration $W_{\bullet}$, give a mixed Hodge structure on $H^n(Y_0)$, which is pure for $n\neq 3$. More precisely, 
 \begin{enumerate} 
 \item[\rm(i)]  $H^n(Y_0) =0$ for $n=1,5$;
 \item[\rm(ii)] For $n=2k$, the mixed Hodge structure on  $H^{2k}(Y_0)$ is pure and $H^{2k}(Y_0) = H^{k,k}(Y_0)$;
 \item[\rm(iii)] As mixed Hodge structures over $\Q$, $H^3(Y_0)$ is an extension of the pure Hodge structure $H^3(X_0)$ by a pure weight two piece  $\cong \Q(-1)$.
 \end{enumerate}
\end{theorem} 
\begin{proof} Although we have not necessarily assumed that $X_0$ is K\"ahler, its cohomology satisfies the $\partial\bar{\partial}$-lemma and the same is true for the projective varieties $E_i$ and $Q_i$. Thus all of the terms in the Mayer-Vietoris sequence carry pure Hodge structures and the morphisms are  morphisms of Hodge structures. Then the method of proof of \cite[(4.2)]{GriffithsSchmid} shows that there is a mixed Hodge structure on $H^n(Y_0)$, and the usual arguments with mixed Hodge complexes (\cite[(8.1.9)]{DeligneHodgeIII} or \cite[Theorem 3.18]{PetersSteenbrink}) show  that the above spectral sequence degenerates at $E_1$.

The other statements are proved by explicit calculation.
Starting with $\scrO_{Y_0}$, we have the usual resolution
$$0 \to \scrO_{Y_0} \to \nu_*(\scrO_{\widetilde{X}_0} \oplus \bigoplus_i\scrO_{E_i} ) \to \bigoplus_i(j_i)_*\scrO_{Q_i} \to 0.$$
It follows that $H^0(Y_0; \scrO_{Y_0}) \cong \Cee$, $H^3(Y_0; \scrO_{Y_0}) \cong H^3(X_0; \scrO_{X_0}) \cong \Cee$, and $H^k(Y_0; \scrO_{Y_0}) =0$, $k\neq 0,3$. As for $\Omega_{Y_0}^1/\tau_{Y_0}^1$, beginning with the exact sequence 
$$0 \to \Omega_{Y_0}^1/\tau_{Y_0}^1 \to \nu_*\Omega_{\widetilde{Y}_0}^1\to \bigoplus _i(j_i)_*\Omega_{Q_i}^1\to 0,$$
we see that $H^0(Y_0; \Omega_{Y_0}^1/\tau_{Y_0}^1) = H^3(Y_0; \Omega_{Y_0}^1/\tau_{Y_0}^1)=0$, that $H^1(Y_0; \Omega_{Y_0}^1/\tau_{Y_0}^1)\cong H^2(Y_0)$ and that there is an exact sequence
$$0 \to \Cee \to H^2(Y_0; \Omega_{Y_0}^1/\tau_{Y_0}^1) \to H^2(X_0; \Omega_{X_0}^1) \to 0.$$
The cases $H^q(Y_0; \Omega_{Y_0}^p/\tau_{Y_0}^p)$, $p=2,3$ are analyzed in a similar way. We remark that, by directly checking all possible cases for all $k$, it follows that
$$\sum_{p+q=k}\dim H^q(Y_0; \Omega_{Y_0}^p/\tau_{Y_0}^p) =\dim H^k(Y_0). $$
Thus we see again that the spectral sequence degenerates at $E_1$. 

The remaining statements also follow by inspection, using the compatibility of the above exact sequences with the Mayer-Vietoris exact sequence. For example, for the Hodge and weight filtrations on $H^3$, there is a surjection from $ \bigoplus _iF^2H^2(Q_i)=0$ to $F^2\cap W_2$, so that $F^2\cap W_2 =0$, and similarly  $F^1\cap W_2 = W_2$, i.e.\ $W_2$ is pure of  type $(1,1)$. 
\end{proof}

\subsection{The limiting mixed Hodge structure} We begin by constructing the relative log complex. Recall that $S$ is the base of the deformation $\mathcal{Y}$ of $Y_0$, with discriminant locus $D$, and that $\mathcal{Y}_D \to D$ is the locally trivial part of the deformation of $Y_0$. After shrinking, we will assume that $S$ is a polydisk $\Delta^{h+1}$ and that $D$ is the divisor $\Delta^h\times \{0\}$. Let $S^* = \Delta^h \times \Delta^*$, where $\Delta^*$ is the punctured unit disk, and let $\pi^* \colon\mathcal{Y}^* \to S^*$ be the restriction of $\pi$ to $S^*$. Thus $R^n(\pi^*)_*\Cee = \underline{H}^n$ is a local system over $S^*$.

Define the sheaf  $\Omega^1_{\mathcal{Y}/S}(\log \mathcal{Y}_D)$ by the exact sequence
$$0 \to \pi^*\Omega^1_S(\log D) \to \Omega^1_{\mathcal{Y}}(\log \mathcal{Y}_D) \to  \Omega^1_{\mathcal{Y}/S}(\log \mathcal{Y}_D) \to 0.$$
It is a locally free sheaf of rank $3$. Define the relative log complex via
$$\Omega^\bullet_{\mathcal{Y}/S}(\log \mathcal{Y}_D)  = \bigwedge^\bullet\Omega^1_{\mathcal{Y}/S}(\log \mathcal{Y}_D),$$
with the usual differential.
For a fiber $Y_s$, $s\notin D$, $\Omega^\bullet_{\mathcal{Y}/S}(\log \mathcal{Y}_D)|Y_s \cong \Omega^\bullet_{Y_s}$. For the singular fiber $Y_0$, we set $\Lambda^\bullet_{Y_0} = \Omega^\bullet_{\mathcal{Y}/S}(\log \mathcal{Y}_D)|Y_0$. The complex $\Omega^\bullet_{\mathcal{Y}/S}(\log \mathcal{Y}_D)$ is the relative log complex of Deligne-Steenbrink with extra parameters coming from the locally trivial deformations of $Y_0$. In fact, if $\Delta \to S$ is a morphism of the disk to $S$, transverse to the discriminant locus $D$, then the pullback of $\Omega^\bullet_{\mathcal{Y}/S}(\log \mathcal{Y}_D)$ to $\Delta$ is the usual one parameter relative log complex. The arguments of \cite{Steenbrink} or \cite[Corollary 11.18]{PetersSteenbrink} show:

\begin{theorem} The hypercohomology $\mathbb{H}^n(Y_0; \Lambda^\bullet_{Y_0})$ is isomorphic to the cohomology $H^n(\mathcal{Y}^*\times _{S^*} \widetilde{S ^*};\Cee)$, where $\widetilde{S ^*}= \Delta^h \times \widetilde{\Delta^*}$ is the universal cover of $S^*$. (Here $\widetilde{\Delta^*}\cong \mathfrak{H}$ is the universal cover of $\Delta^*$.) The sheaf 
$$\overline{\mathcal{H}}^n =\mathbb{R}^n\pi_*\Omega^\bullet_{\mathcal{Y}/S}(\log \mathcal{Y}_D)$$
is locally free and satisfies: $\overline{\mathcal{H}}^n|S^*$ is the holomorphic flat vector bundle $\mathcal{H}^n =\underline{H}^n\otimes _\Cee \scrO_{S^*}$, and $\overline{\mathcal{H}}^n$ is Deligne's canonical extension of $\mathcal{H}^n $. \qed
\end{theorem}

The arguments of \cite[Theorem 11.22 and Corollaries 11.23 and 11.24]{PetersSteenbrink} as well as the method of proof of Theorem~\ref{MHS1} then show:

\begin{theorem}\label{MHS2} {\rm(i)} Denote 
$$\mathbb{H}^n(Y_0; \Lambda^\bullet_{Y_0})\cong H^n(\mathcal{Y}^*\times _{S^*} \widetilde{S^*};\Cee)$$ by $H^n_{\text{\rm{lim}}}$. Then there is a mixed Hodge structure on $H^n_{\text{\rm{lim}}}$, the \textsl{limiting mixed Hodge structure}, some of whose properties we recall below.

 \smallskip
 \noindent {\rm(ii)}
The spectral sequence with $E_1$ page 
$$E_1^{p,q} =H^q(Y_0; \Lambda_{Y_0}^p)\implies \mathbb{H}^{p+q}(Y_0; \Lambda_{Y_0}^{\bullet}) =H^{p+q}_{\text{\rm{lim}}}$$
 degenerates at $E_1$ and the corresponding filtration on $H^{p+q}_{\text{\rm{lim}}}$ is the Hodge filtration.
 
  \smallskip
 \noindent {\rm(iii)} Possibly after shrinking $S$, the spectral sequence of coherent sheaves on $S$ whose $E_1$ page is 
 $$E_1^{p,q} = R^q\pi_*\Omega^p_{\mathcal{Y}/S}(\log \mathcal{Y}_D) \implies \mathbb{R}^{p+q}\pi_*\Omega^\bullet_{\mathcal{Y}/S}(\log \mathcal{Y}_D) = \overline{\mathcal{H}}^{p+q}$$
 degenerates at $E_1$. Thus, for $t\in S^*$, the Hodge-de Rham spectral sequence for $Y_t$ degenerates at $E_1$. Moreover, the sheaves $R^q\pi_*\Omega^p_{\mathcal{Y}/S}(\log \mathcal{Y}_D)$ are locally free. \qed
\end{theorem}

In particular, there is a filtration of $\overline{\mathcal{H}}^n$ by holomorphic subbundles $F^\bullet$, which we will somewhat inaccurately call the \textsl{Hodge filtration}. By Lemma~\ref{prelimlemma1}, for $n=3$ this filtration is isotropic over $S^*$ (and in fact over $S$). 

\subsection{The monodromy weight filtration} There is an increasing filtration $V_\bullet$ on the complex $\Lambda^\bullet_{Y_0}$. Because $Y_0$ consists of smooth components meeting transversally along smooth divisors, it takes the following simple form
$$0 \to V_0 \to V_1 = \Lambda^\bullet_{Y_0} \to V_1/V_0 \to 0.$$
Here $V_0 \cong \Omega_{Y_0}^{\bullet}/\tau_{Y_0}^{\bullet}$ and $V_1/V_0  \cong \bigoplus_i (j_i)_*\Omega^{\bullet -1}_{Q_i}$, by \cite[(3.5)]{Friedman1983} or \cite[11.2.5]{PetersSteenbrink}. By  \cite[Theorem 11.29]{PetersSteenbrink} (and the discussion prior to the statement), we have

\begin{theorem} The homomorphism 
$$\mathbb{H}^n(Y_0; \Omega_{Y_0}^{\bullet}/\tau_{Y_0}^{\bullet}) \to \mathbb{H}^n(Y_0;\Lambda^\bullet_{Y_0})$$
is the specialization homomorphism $H^n(Y_0; \Cee) \to H^n_{\text{\rm{lim}}}$,
and it is a morphism of mixed Hodge structures. \qed
\end{theorem}
Consider now the long exact sequence associated to the short exact sequence $0\to V_0 \to V_1\to V_1/V_0 \to 0$. In particular, we get the two exact sequences of mixed Hodge structures (all groups with $\Cee$-coefficients) 
$$ 0 \to H^1(Y_0) \to  H^1_{\text{\rm{lim}}}\to \bigoplus_iH^0(Q_i)(-1) \to H^2(Y_0) \to H^2_{\text{\rm{lim}}} \to 0$$
and
$$0 \to H^3(Y_0) \to H^3_{\text{\rm{lim}}} \to \bigoplus_iH^2(Q_i)(-1) \to H^4(Y_0) \to H^4_{\text{\rm{lim}}} \to 0.$$
  The map $\bigoplus_iH^0(Q_i)(-1) \to H^2(Y_0)$ is injective, since the composite map
$$\bigoplus_iH^0(Q_i)(-1) \to H^2(Y_0) \to H^2(\widetilde{X}_0)\oplus \bigoplus_iH^2(E_i)$$
is injective (it restricts to  the Gysin map $H^0(Q_i)(-1) \to  H^2(E_i)$ on each summand). Thus $H^2_{\text{\rm{lim}}}$ has dimension $b-r+1$, and the same must be true for  $H^4_{\text{\rm{lim}}}$. Then, since $H^4(Y_0)$ has dimension $b+r$ and the dimension of $H^4_{\text{\rm{lim}}}$ is $b-r+1$,  the image of $\bigoplus_iH^2(Q_i)(-1)$ in $H^4(Y_0)$ has dimension $2r-1$ and hence the kernel of this map has dimension one. Explicitly, it is easy to check that the kernel of $\bigoplus_iH^2(Q_i)(-1) \to H^4(Y_0)\subseteq H^4(\widetilde{X}_0)\oplus \bigoplus_iH^4(E_i)$ is identified with
$$\left\{(m_1([\sigma_1]-[f_1]),\dots,  m_r([\sigma_r]-[f_r])):\sum_im_i[C_i] = 0\right\}.$$

Summarizing,

\begin{theorem}\label{MSH3} {\rm(i)} $H^1_{\text{\rm{lim}}}=H^5_{\text{\rm{lim}}}=0$.

  \smallskip
 \noindent {\rm(ii)} $H^2_{\text{\rm{lim}}}$ and $H^4_{\text{\rm{lim}}}$ are pure of weights two and four respectively and dimension $b-r+1$, with $H^2_{\text{\rm{lim}}}=  H^{1,1}_{\text{\rm{lim}}}$ and $H^4_{\text{\rm{lim}}}=  H^{2,2}_{\text{\rm{lim}}}$.
 
   \smallskip
 \noindent {\rm(iii)} There is an exact sequence of mixed Hodge structures
 $$0 \to H^3(Y_0) \to H^3_{\text{\rm{lim}}}  \to \Q(-2) \to 0.$$
 Thus the weight filtration on  $H^3_{\text{\rm{lim}}}$ is given by
 $$0 \subseteq W_2 \subseteq W_3 \subseteq W_4 = H^3_{\text{\rm{lim}}},$$
 where $W_3= H^3(Y_0)$, $W_3/W_2 \cong H^3(X_0)$, $W_2\cong \Q(-1)$ and $W_4/W_3\cong \Q(-2)$. \qed
 \end{theorem}

 \begin{remark} (1) Somewhat more general formulas for $b_2(Y_s)$ and $b_3(Y_s)$ are given in \cite[Lemma 8.1]{Friedman1991} by comparing the Mayer-Vietoris sequences for $X_0$ and $X_t$. 
 
 \smallskip
 \noindent (2) In our main case of interest, the classes $[C_i]$ span $H^2(X_0)$ and satisfy one linear relation. Hence $b = r-1$ and thus $H^2_{\text{\rm{lim}}}=0$, i.e.\ $H^2(Y_t; \Zee)$ is torsion for $t\notin D$. 
 \end{remark}
 
 An easy argument using Theorem~\ref{MHS2} then shows:
 
 \begin{corollary} If $Y_s$ is a small smoothing of $Y_0$, then $H^i(Y_s; \scrO_{Y_s}) =H^0(Y_s; \Omega_{Y_s}^i) = 0$ for $i=1,2$. \qed
 \end{corollary}

 We relate the weight filtration to the monodromy filtration on $H^3$ as follows. Let $T$ be the monodromy of the family acting on $H^3$ and let $N = T-I$. Thus $N$ is a nilpotent matrix, and in fact $N^2 =0$. More precisely,
 
 \begin{theorem} $\Ker N = W_3 = \im H^3(Y_0)$ and $\im N = W_2$.
 \end{theorem}
 \begin{proof} By general theory \cite[Theorem 11.28]{PetersSteenbrink}, $N$ is a morphism of mixed Hodge structures of type $(-1, -1)$, and hence $W_3 \subseteq \Ker N$ and $\im N \subseteq W_2$. Thus $N$ induces a homomorphism of one dimensional $\Q$--vector spaces $W_4/W_3 \to W_2$.  To see the statement of the theorem, it therefore   suffices to prove that $N\neq 0$, or equivalently that $T\neq I$. This follows from Picard-Lefschetz theory: associated to each double point $p_i$, is a vanishing cycle $\xi_i$, viewed as an element of cohomology. By assumption, there exists a $\xi\in H^3(Y_t;\Zee)$ of infinite order such that each $\xi_i$ is a multiple $r_i\xi$ of $\xi$ and the $\Q$-span of the $\xi_i$ is equal to $\Q\cdot \xi$, so that not all of the $r_i$ can be $0$. By the Picard-Lefschetz formula,
 $$T(\alpha) = \alpha + \sum_i2 \langle \alpha, \xi_i\rangle \xi_i = \alpha + \left(\sum_i2r_i^2 \right)\langle \alpha, \xi \rangle \xi,$$
 where the $2$ reflects the base change of order $2$ in the passage from deformations of $\overline{X}_0$ to deformations of the semistable model $Y_0$.  Thus there exists a positive rational number $r$ such that  $$T(\alpha) = \alpha + r\langle \alpha, \xi\rangle \xi$$
and so $T\neq I$.  

We can give a direct argument that $N\colon W_4/W_3\to W_2$ is an isomorphism as follows.  The action of $N$ on the graded pieces $W_4/W_3 \to W_2$ is calculated in  \cite[11.2.5]{PetersSteenbrink}, and one checks (cf.\ \cite[\S11.3]{PetersSteenbrink}) that it is  the  homomorphism (induced by $\pm \Id \colon \bigoplus _iH^2(Q_i) \to \bigoplus _iH^2(Q_i)$):
\begin{gather*}
\Ker\left( \bigoplus _iH^2(Q_i) \to H^4(\widetilde{X}_0) \oplus \bigoplus_iH^4(E_i)\right)\\
\to \Coker\left(H^2(\widetilde{X}_0) \oplus \bigoplus_iH^2(E_i) \to  \bigoplus _iH^2(Q_i)\right).
\end{gather*}
To see that  $N \colon W_4/W_3 \to W_2$ is an isomorphism, using the comments before Proposition~\ref{prop2.4} and Theorem~\ref{MSH3}, it suffices to show that, if $(m_1, \dots, m_r)\in \Q^r$ is a nonzero vector such that $\sum_im_i[C_i] = 0$ in $H^4(X_0)$, then $(m_1, \dots, m_r)$ is not in the subspace 
$$I= \{((\xi\cdot [C_1]),  \dots, (\xi\cdot [C_r]):\xi\in H^2(X_0)\}.$$
But $(m_1, \dots, m_r)$ is orthogonal to $I$ under the standard inner product on $\Q^r$, so that $(m_1, \dots, m_r)\in I$ $\implies$ $(m_1, \dots, m_r) =0$.
 \end{proof}
 
 \begin{proposition} With the alternating nondegenerate pairing $\langle \cdot, \cdot \rangle$ on $H^3_{\text{\rm{lim}}} \cong H^3(Y_t)$, $W_2^\perp = W_3$. Hence, if $\xi$ is a generator for $W_2$ and $\eta$ generates $W_4/W_3$, then $\langle \xi, \eta \rangle\neq 0$.
 \end{proposition}
 \begin{proof} The first statement is clear since $N(\alpha) = r\langle \alpha, \xi\rangle \xi$, with $\xi \neq0$, so that $\im N = \Cee \cdot \xi$ and $\Ker N = \xi^\perp$. (It also follows from the fact that $\langle N(\alpha), \beta \rangle = -\langle \alpha , N(\beta) \rangle$.)
 The final statement follows because $\langle \cdot, \cdot \rangle$ is nondegenerate. 
 \end{proof}

\subsection{The differential of the period map} The flat vector bundle $\mathcal{H}^3$ has an integrable connection $\nabla$ and a decreasing  filtration $F^\bullet$ by holomorphic subbundles. Moreover, for every $s\in S^*$, the associated graded 
$$F^p_s/F^{p+1}_s \cong H^{3-p}(Y_s; \Omega^p_{Y_s}).$$
In any small simply connected open subset $U$ of $S^*$, or on the universal cover $\widetilde{S^*}$, the restriction or pullback of $\mathcal{H}^3$ is canonically trivialized by $\nabla$. Given such a trivialization, we define the \textsl{period map} on $U$ to be the holomorphic map from $U$ to an appropriate flag manifold defined by sending $s\in U$ to the subspaces $F^p_s$ of $H^3(Y_s)$. By openness of versality, the tangent space to $S^*$ at $s$ is identified with $H^1(Y_s; T_{Y_s})$; more precisely, the Kodaira-Spencer map $T_{S^*,s}\to H^1(Y_s; T_{Y_s})$ is an isomorphism. The standard arguments in the K\"ahler case (see e.g.\ \cite[Proposition 7.7]{DeligneRSP}) show that the differential of the period map is computed at the point $s$ via the natural homomorphism
$$H^1(Y_s; T_{Y_s}) \to \bigoplus _p\Hom(H^{3-p}(Y_s; \Omega^p_{Y_s}), H^{3-p+1}(Y_s; \Omega^{p-1}_{Y_s}))$$ 
given by  cup product and contraction. A similar statement holds globally: the differential of the period map is given by the homomorphism induced by cup product:
$$R^1\pi_*T_{\mathcal{Y}^*/S^*} \to  \bigoplus _p\Hom(R^{3-p}\pi_* \Omega^p_{\mathcal{Y}^*/S^*} , R^{3-p+1}\pi_*\Omega^{p-1}_{\mathcal{Y}^*/S^*}).$$
Since $\Omega^3_{Y_t} \cong \scrO_{Y_t}$, the cup product homomorphism
$$H^1(Y_t; T_{Y_t}) \to \Hom(H^0(Y_t; \Omega^3_{Y_t}), H^1(Y_t; \Omega^2_{Y_t}))\cong H^1(Y_t; \Omega^2_{Y_t})$$ 
is an isomorphism. Similarly, after trivializing the line bundle $R^0\pi_* \Omega^3_{\mathcal{Y}^*/S^*}$, i.e.\ after choosing an everywhere generating  section of $\Omega^3_{\mathcal{Y}^*/S^*}$, the cup product homomorphism
$$R^1\pi_*T_{\mathcal{Y}^*/S^*} \to \Hom(R^0\pi_* \Omega^3_{\mathcal{Y}^*/S^*} , R^1\pi_*\Omega^2_{\mathcal{Y}^*/S^*})\cong R^1\pi_*\Omega^2_{\mathcal{Y}^*/S^*}$$
is an isomorphism.

\section{The variational argument}

\subsection{The basic setup} We begin by abstracting the situation of \S2. Let $H$ be a vector space with a nondegenerate  alternating bilinear form $\langle \cdot, \cdot \rangle$ and a standard symplectic basis $e_0, \dots, e_{h+1}, f_0, \dots, f_{h+1}$ (i.e.\ $\langle e_i, f_j \rangle =\delta_{ij}$, and $\langle e_i, e_j \rangle =\langle f_i, f_j \rangle = 0$ for all $i,j$). We assume that $H$ is in fact defined over $\Q$, i.e.\ is the complexification of a $\Q$-vector space $H_\Q$, and that the above basis is a $\Q$-basis. In particular, $H$ is defined over $\Ar$ so that complex conjugation is defined on $H$. Let $N\colon H \to H$ be the rational linear map defined by: $N(e_i) = 0$ for all $i$, $N(f_i) =0$ for $i\neq h+1$, and $N(f_{h+1}) = e_{h+1}$. Then
$$\langle N(\alpha), \beta\rangle + \langle \alpha, N(\beta)\rangle = 0$$
for all $\alpha, \beta \in H$. Define
$$W_2 = \Cee e_{h+1} \subseteq W_3 = \text{  span } \{e_0, \dots, e_{h+1}, f_0, \dots,f_h\} \subseteq W_4 = H.$$

Let $S = \Delta^h\times \Delta$, with coordinates $t_1, \dots, t_h, q$,  let $S^* = \Delta^h\times \Delta^*\subseteq S$, and let $D= \Delta^h\times\{0\}$. We shall abbreviate $(t_1, \dots,t_h, q)$ by $(t,q)$. Write $q = e^{2\pi \sqrt{-1}z}$, where $z$ is the usual coordinate on the upper half plane $\mathfrak{H} = \widetilde{\Delta ^*}$; equivalently,
$$z = \frac{\log q}{2\pi \sqrt{-1}}.$$
Let   $\varphi \colon \widetilde{S ^*}=\Delta^h\times\mathfrak{H}\to S^*$ be the universal cover map. 
Setting $T =\exp N$ defines an action of $\pi_1(S^*) \cong \Zee$ on $H$, where $1$ acts as $T$,  and hence a local system $\underline{H}$ over $S^*$. Let $\mathcal{H}= \underline{H} \otimes _\Cee\scrO_{S^*}$ be the corresponding holomorphic vector bundle over $S^*$ and $\overline{\mathcal{H}}$ the canonical extension of $\mathcal{H}$ to $S$. By \cite{DeligneRSP}, we can take $\overline{\mathcal{H}} \cong H\otimes_\Cee \scrO_S$, the trivial holomorphic vector bundle over $S$ with fiber $H$, with the meromorphic connection $\nabla$ whose associated connection $1$-form is 
$\displaystyle  -\frac{N}{2\pi \sqrt{-1}}\frac{dq}{q}$. 
The bundle $\varphi^*\mathcal{H}$ is trivialized by $\nabla$ and the fiber at any point of $\widetilde{S^*}$ is identified with $H$. The fiber of $\mathcal{H}$ at any point is identified with $H$ modulo the action of $\{T^k=\exp(kN): k\in \Zee\}$. The fiber of $\overline{\mathcal{H}}$ over $0\in D$ is identified with $H$ up to the action of the unipotent group $\{\exp(\lambda N): \lambda\in \Cee\}$. 
The local flat sections of $\mathcal{H}$ over $S^*$ are then sections locally of the form $\exp(zN)v$, where $v\in H$. A holomorphic section $\sigma$ of $\mathcal{H}$, viewed as a holomorphic section $\sigma$ of the trivial bundle $\varphi^*\mathcal{H}\cong H\otimes _\Cee \scrO_{\widetilde{S^*}}$ with the invariance property $\sigma(t, z+1) = T\sigma(t,z)$,  extends to a holomorphic section of $\overline{\mathcal{H}}$ if and only if the   section $\exp(-zN)\varphi^*\sigma$, viewed as a holomorphic section of $\varphi^*\mathcal{H}$, extends to a single-valued holomorphic function from $S$ to $H$. Given a holomorphic section $\sigma$ of $\mathcal{H}$, we denote $\nabla_{\partial /\partial t_i}\sigma$ by $\displaystyle \frac{\partial \sigma}{\partial t_i}$, and similarly for the coordinate $q$.   

Finally, we are given a filtration of $\overline{\mathcal{H}}$ by holomorphic subbundles $F^\bullet$. It satisfies: 
\begin{itemize}
\item[\rm(i)] $F^3$ is a line bundle, hence $F^3 = \scrO_S\cdot \tilde\omega(t,q)$ for some nowhere vanishing holomorphic function $\tilde\omega(t,q)$ with values in $H$. 
\item[\rm(ii)] Over $S^*$, a basis for $F^2|S^*$ is given by
$$\tilde\omega, \frac{\partial \tilde\omega}{\partial t_1}, \dots , \frac{\partial \tilde\omega}{\partial t_h}, \frac{\partial \tilde\omega}{\partial q}.$$
 We can also replace the last term  $\displaystyle\frac{\partial \tilde\omega}{\partial q}$ by $\displaystyle\frac{\partial \tilde\omega}{\partial z}$ on $\varphi^*\mathcal{H}$, since
 $$\frac{1}{2\pi \sqrt{-1}} \frac{\partial}{\partial z} = q\frac{\partial}{\partial q}.$$ 
\item[\rm(iii)] (First Hodge-Riemann bilinear relation) With respect to the form $\langle \cdot, \cdot \rangle$, $(F^3)^\perp = F^1$ and $(F^2)^\perp = F^2$, so that $F^2$ is a maximal isotropic subbundle.
\item[\rm(iv)] For $s\in D = \Delta^h\times \{0\}$, the filtrations $F^\bullet_s$ and $W_\bullet$ define a mixed Hodge structure on $H$, with $W_2\cong \Q(-1)$, $W_4/W_3\cong \Q(-2)$, and $W_3/W_2$ is a pure weight three Hodge structure with $h^{3,0} = h^{0,3} = 1$, and hence $h^{2,1} = h^{1,2} = h$.
\end{itemize}

As a consequence, we record the following facts:

\begin{lemma}\label{Fprops} Under the above assumptions,
\begin{itemize}
\item[\rm(i)] The subbundle $F^2$ has rank $h+2$.
\item[\rm(ii)] For $s\in D$, $F^3_s\subseteq W_3$ and $F^2_s+W_3 = W_4$. Equivalently, there exists a $v\in F^2_s$ such that, writing $v$ as a linear combination of the $e_i, f_i$,  the coefficient of $f_{h+1}$ in $v$ is $1$.
\item[\rm(iii)] For $s\in D$, $F^2_s\cap W_2 = 0$.  \qed
\end{itemize}
\end{lemma}

\subsection{The bundle $\overline{\mathcal{H}}_\#$} By the above, $e_{h+1}$ defines a global holomorphic section of $\mathcal{H}$ and of $\overline{\mathcal{H}}$. We define $\overline{\mathcal{H}}_\# = (e_{h+1})^\perp/\scrO_S\cdot e_{h+1}$. It is a flat vector bundle canonically isomorphic to $H_\#\otimes _\Cee\scrO_S$, where 
$$H_\# = (e_{h+1})^\perp/\Cee\cdot e_{h+1}= W_3/W_2.$$ 
Here we take $(e_{h+1})^\perp$ inside the vector space $H$, not the holomorphic bundle $\overline{\mathcal{H}}$. For each $s\in D$, the filtration $F^\bullet_s$ induces a pure weight three Hodge structure on $H_\#$.
The bundle  $\overline{\mathcal{H}}_\#$ has rank $2h+2$.  There is an induced nondegenerate alternating bilinear form $\langle \cdot, \cdot \rangle$  on $H_\#$ and on $\overline{\mathcal{H}}_\#$.  

Define $F^2_\#$ to be the image of $F^2\cap (e_{h+1})^\perp$ in $\overline{\mathcal{H}}_\#$.

\begin{lemma}\label{Fsharp} Possibly after shrinking $S$, $F^2_\#$ is a holomorphic isotropic subbundle of $\overline{\mathcal{H}}_\#$ of rank $h+1$, and, as $C^\infty$ bundles, $F^2_\# \oplus \overline{F}^2_\#\cong H_\#\otimes_\Cee C^\infty_S$.
\end{lemma} 
\begin{proof} By (ii) of Lemma~\ref{Fprops}, for all $s\in D$, $\langle e_{h+1}, F^2_s\rangle\neq 0$. Thus, possibly after shrinking $S$, we can assume that $F^2\cap (e_{h+1})^\perp$ is a holomorphic subbundle of $\overline{\mathcal{H}}$ of rank $h+1$. By (iii) of Lemma~\ref{Fprops}, for all $s\in D$, $F^2_s \cap \Cee\cdot e_{h+1}=0$. Thus, again possibly after shrinking $S$, we can assume that the projection $F^2\cap (e_{h+1})^\perp \to \overline{\mathcal{H}}_\#$ is injective and of maximal rank at every point of $S$. It follows that $F^2_\#$ is a holomorphic  subbundle of $\overline{\mathcal{H}}_\#$ of rank $h+1$, and it is isotropic (i.e.\ $\langle F^2_\#, F^2_\#\rangle =0$) because $F^2$ is isotropic. 

Finally, for $s\in D$, $(F^2_\#)_s \oplus (\overline{F}^2_\#)_s\cong H_\#$ because, for each $s\in S$, $H_\#$ carries a weight 3 Hodge structure for which $(F^2_\#)_s$ is the corresponding piece of the Hodge filtration. After shrinking $S$, we can assume that, for all $s\in S$, $(F^2_\#)_s \oplus (\overline{F}^2_\#)_s\cong H_\#$. Hence $F^2_\# \oplus \overline{F}^2_\#\cong H_\#\otimes_\Cee C^\infty_S$.
\end{proof}

\subsection{Normalizing the holomorphic form} Begin by choosing an arbitrary holomorphic, nowhere vanishing section $\tilde{\omega}$ of the line bundle $F^3$. We can write (using the basis of flat sections $e_0, \dots, e_{h+1}, f_0, \dots, f_{h+1}$ of $\varphi^*\mathcal{H}$)
$$\varphi^*\tilde{\omega}(t,z) =  \sum_{i=0}^{h+1}\widetilde{A}_ie_i + \sum_{i=0}^{h+1}\widetilde{B}_if_i,$$
where the $\widetilde{A}_i$, $\widetilde{B}_i$ are holomorphic in $t_1, \dots, t_h, z$. The invariance property, that $\tilde{\omega}$ defines a holomorphic section of $\overline{\mathcal{H}}$, gives: $\widetilde{A}_0, \dots, \widetilde{A}_h$ and $\widetilde{B}_0, \dots, \widetilde{B}_{h+1}$ are holomorphic functions of $t$ and $q$ on $S$, and
$$\widetilde{A}_{h+1}= C(t,q) + z\widetilde{B}_{h+1}(t,q),$$
where $C(t,q)$ is a holomorphic function of $t$ and $q$ on $S$. Equivalently, viewed as a holomorphic section of the bundle $\overline{\mathcal{H}} \cong H\otimes_\Cee \scrO_S$,
$$\tilde{\omega}(t,z) =  \sum_{i=0}^h\widetilde{A}_ie_i + Ce_{h+1} +\sum_{i=0}^{h+1}\widetilde{B}_i f_i.$$

In the limit (i.e.\ for $s\in D$), $F^3_s \subseteq  W_3$ and hence $\widetilde{B}_{h+1}(t,0) = 0$. Nonetheless:

\begin{lemma} The coefficient $\widetilde{B}_{h+1}$ is not identically $0$.
\end{lemma} 
\begin{proof} Suppose instead that $\widetilde{B}_{h+1}$ is identically $0$, so that $\tilde{\omega}$ lies in the (flat) subbundle $(e_{h+1})^\perp$ of $\overline{\mathcal{H}}$. Then, over $S^*$, the sections
$$\tilde\omega, \frac{\partial \tilde\omega}{\partial t_1}, \dots , \frac{\partial \tilde\omega}{\partial t_h}, \frac{\partial \tilde\omega}{\partial q}$$ 
all lie in $(e_{h+1})^\perp$. It follows that $F^2|S^*$ lies in $(e_{h+1})^\perp$, and hence so does $F^2$. But this contradicts (ii) of Lemma~\ref{Fprops}. 
\end{proof}

We define the normalized meromorphic section $\omega$ of $F^3$ by dividing  $\tilde{\omega}$ by $\widetilde{B}_{h+1}$. Thus $\omega = (\widetilde{B}_{h+1})^{-1}\tilde{\omega}$ and
$$\varphi^*\omega = \sum_{i=0}^hA_ie_i + (A' + z)e_{h+1} + \sum_{i=0}^hB_if_i + f_{h+1},$$
where $A_i$, $B_i$, and $A'$ are meromorphic functions of $t$ and $q$ on $S$. We write this as 
$$\varphi^*\omega =\psi + ze_{h+1}.$$

Henceforth we shall ignore the $\varphi^*$ and view $\omega$ and its derivatives as meromorphic functions either on $\widetilde{S^*}$ or on $S^*$. Note that 
$$\frac{\partial \omega}{\partial t_i} = (\widetilde{B}_{h+1})^{-1}\frac{\partial \tilde{\omega}}{\partial t_i} + \frac{\partial}{\partial t_i}(\widetilde{B}_{h+1})^{-1}\cdot\tilde{\omega},$$
and similarly for the partial derivatives with respect to $q$ or $z$.  Thus, over the nonempty open subset of $S^*$ where $\widetilde{B}_{h+1} \neq 0$, the span of $\omega$ and its derivatives with respect to $t_1, \dots, t_h, q$, or equivalently with respect to $t_1, \dots, t_h, z$, is the holomorphic subbundle $F^2$. 

\subsection{The main calculation} First, a preliminary definition:

\begin{definition} A \textsl{real meromorphic function} on $S$ is an element of the field of quotients $\mathcal{K}(S)$ of the ring of (complex valued) real analytic functions on $S$ (which is an integral domain). A nonzero real meromorphic function on $S$  is defined and real analytic on an open dense subset of $S$. Real meromorphic functions on $S^*$ are defined similarly, and we let $\mathcal{K}(S^*)$ be the field of all such. In particular $\mathcal{K}(S)$ is a subfield of $\mathcal{K}(S^*)$.
\end{definition}

The function $\log |q|$ is real analytic on $S^*$, hence is an element of $\mathcal{K}(S^*)$. However:

\begin{lemma}\label{notrealmero} The function $\log |q|$ is not a real meromorphic function on $S$.
\end{lemma}
\begin{proof} Write $q= re^{\sqrt{-1}\cdot \theta}$. Suppose that $\log|q|=\log r$ is of the form $F/G$ where $F$ and $G$ are real analytic functions on $S$ and $G\neq 0$. Choose values of $\theta$ and $t$ so that $G(t, re^{\sqrt{-1}\cdot \theta})$ is not identically zero and is convergent at $r=0$ as a power series in $r$. Then, for $0< r \ll 1$,  $\log r=f(r)/g(r)$, where $f(r)$, $g(r)$ are convergent power series in $r$ (at $r=0$) and $g(r)$ is not identically $0$. Thus $g(r) = r^ag_0(r)$ for some nonnegative integer $a$, where $g_0(0) \neq 0$. Hence there exists  a nonnegative  integer $a$ such that $r^a\log r$ extends to a $C^\infty$ function in some interval around $r=0$. This is a contradiction, since the $a^{\text{th}}$ derivative of $r^a\log r$ is unbounded at $0$.
\end{proof}

Our goal now is to prove:

\begin{theorem}\label{maincalc} With $\omega$ the normalized meromorphic section of $F^3$ given above, there exist real meromorphic functions $M_1$ and $M_2$ on $S$ with $M_1\neq 0$, such that, as an element of $\bigwedge^{2h+4}H\otimes_\Cee\mathcal{K}(S^*)$,
\begin{gather*}
\omega\wedge \bar\omega \wedge \frac{\partial \omega}{\partial z} \wedge \overline{\frac{\partial \omega}{\partial z}} \wedge  \frac{\partial \omega}{\partial t_1} \wedge \overline{\frac{\partial \omega}{\partial t_1}} \wedge  \cdots \wedge\frac{\partial \omega}{\partial t_h} \wedge \overline{\frac{\partial \omega}{\partial t_h}} =\\
=(  (z-\bar{z}) M_1  + M_2)(e_0\wedge \cdots \wedge e_{h+1}\wedge f_0\wedge \cdots \wedge f_{h+1}).
\end{gather*}
\end{theorem}

First, we show how to apply Theorem~\ref{maincalc} toward establishing  the $\partial\overline{\partial}$-lemma: 

\begin{corollary}\label{maincor} There exists a nonempty open dense subset of $S^*$, the complement of a proper real analytic subvariety in $S^*$,  such that, for all $s\in S^*$,
$$F^2_s \oplus \overline{F}^2_s \cong H.$$ 
\end{corollary}
\begin{proof} By Theorem~\ref{maincalc}, if $F^2_s $ and $\overline{F}^2_s$ do not span $H$ on an open subset where $M_1$, $M_2$, and $\widetilde{B}_{h+1}^{-1}$ are defined, then $(z-\bar{z}) M_1  + M_2$ is identically $0$. We have 
$$z-\bar{z} = \frac{1}{\pi\sqrt{-1}}\log r = \frac{1}{\pi\sqrt{-1}}\log |q|.$$
 Thus $\log |q| = -\pi\sqrt{-1}M_2/M_1$ is a real meromorphic function on $S$, contradicting Lemma~\ref{notrealmero}.  
  
 Hence $(z-\bar{z}) M_1  + M_2$ is not identically $0$ on $S^*$. Let $U$ be the nonempty open dense subset  of $S^*$ where  $M_1$, $M_2$, and $\widetilde{B}_{h+1}^{-1}$ are defined, and for which   $(z-\bar{z}) M_1  + M_2$ does not vanish. Then $U$ is the complement of a proper real analytic subvariety in $S^*$.   For $s\in U$, we have $F^2_s \oplus \overline{F}^2_s \cong H$ as claimed.
\end{proof}

Combining Corollary~\ref{maincor} and Corollary~\ref{prelimcor}, we obtain:

\begin{corollary}\label{firstreduction} Let $\pi\colon \mathcal{Y}\to S$ be as in \S2.3, \S2.5. There exists a nonempty open dense subset of $S^*$, the complement of a proper real analytic subvariety in $S^*$,  such that, for all $s\in S^*$, the fiber $Y_s$ satisfies the $\partial\overline{\partial}$-lemma. \qed  
\end{corollary}

\begin{proof}[Proof of Theorem~\ref{maincalc}] Write $\omega = \psi + ze_{h+1}$, where $\psi$ is a meromorphic section of $\overline{\mathcal{H}}$, i.e.\  whose coordinates are meromorphic functions of $(t,q)$, and such that $\langle \psi, e_{h+1}\rangle = 1$, i.e.\ the coefficient of $f_{h+1}$ in $\psi$ is $1$. Thus
$$ \omega\wedge \bar{\omega} = \bar{z}(\psi \wedge e_{h+1}) - z (\bar{\psi}\wedge e_{h+1})  
+(\psi \wedge \bar{\psi}).$$
Then $\bar\omega = \bar{\psi} + \bar{z}e_{h+1}$. Taking derivatives, we have
$$\frac{\partial \omega}{\partial z} = \frac{\partial \psi}{\partial z}+ e_{h+1}.$$
Here, since $\displaystyle \frac{1}{2\pi \sqrt{-1}} \frac{\partial}{\partial z} = q\frac{\partial} {\partial q}$, $\displaystyle  \frac{\partial \psi}{\partial z}$ and $\displaystyle  \frac{\partial \omega}{\partial z}$ are meromorphic sections of $\overline{\mathcal{H}}$ (their coefficients are meromorphic functions of $(t,q)$), and the coefficient of $f_{h+1}$ in each is $0$. Similarly 
$$\overline{\frac{\partial \omega}{\partial z}}=  \overline{\frac{\partial \psi}{\partial z}}+ e_{h+1}.$$
Computing, we see that
\begin{gather*}
\Xi = \omega\wedge \bar{\omega} \wedge \frac{\partial \omega}{\partial z} \wedge \overline{\frac{\partial \omega}{\partial z}} =\\
= [\bar{z}(\psi \wedge e_{h+1}) - z (\bar{\psi}\wedge e_{h+1})+ (\psi \wedge \bar{\psi})] \wedge \left(\frac{\partial \psi}{\partial z} \wedge e_{h+1}  -  \overline{\frac{\partial \psi}{\partial z}}\wedge e_{h+1} + \frac{\partial \psi}{\partial z}\wedge\overline{\frac{\partial \psi}{\partial z}}\right)\\
=\bar{z}\left(\psi \wedge e_{h+1}\wedge\frac{\partial \psi}{\partial z}\wedge\overline{\frac{\partial \psi}{\partial z}}\right)  -z\left(\bar{\psi}\wedge e_{h+1}\wedge \frac{\partial \psi}{\partial z}\wedge\overline{\frac{\partial \psi}{\partial z}}\right)\\
+ \psi \wedge \bar{\psi}  \wedge \left(\frac{\partial \psi}{\partial z} \wedge e_{h+1}  -  \overline{\frac{\partial \psi}{\partial z}}\wedge e_{h+1} + \frac{\partial \psi}{\partial z}\wedge\overline{\frac{\partial \psi}{\partial z}}\right).
\end{gather*}
Setting 
$$\Phi = \psi \wedge \bar{\psi}  \wedge \left(\frac{\partial \psi}{\partial z} \wedge e_{h+1}  -  \overline{\frac{\partial \psi}{\partial z}}\wedge e_{h+1} + \frac{\partial \psi}{\partial z}\wedge\overline{\frac{\partial \psi}{\partial z}}\right),$$
we can write the expression above, as 

\begin{equation*}
\Xi = (z-\bar{z})\left(\frac{\partial \psi}{\partial z}\wedge\overline{\frac{\partial \psi}{\partial z}}\right)\wedge e_{h+1}\wedge f_{h+1}+   \Phi + \dots, \tag{$*$}
\end{equation*}
 where the remaining terms do not involve $f_{h+1}$ (but might involve $z$ or $\bar{z}$).  
 
Consider the wedge product  
$$\Psi = \frac{\partial \psi}{\partial z}\wedge\overline{\frac{\partial \psi}{\partial z}}\wedge\frac{\partial \omega}{\partial t_1} \wedge \overline{\frac{\partial \omega}{\partial t_1}} \wedge  \cdots \wedge\frac{\partial \omega}{\partial t_h} \wedge \overline{\frac{\partial \omega}{\partial t_h}}.$$
Note that none of the terms in the wedge product  involve $f_{h+1}$. In fact, we have the following:

\begin{lemma} There exists a nonzero real meromorphic function $M_1$ on $S$ such that  $\Psi$ is of the form
$$M_1e_0\wedge f_0\wedge \cdots \wedge e_h\wedge f_h + \Omega\wedge e_{h+1},$$
for some  $\Omega \in \mathcal{K}(S)\otimes _\Cee\bigwedge^{2h+1}H$. 
\end{lemma}
\begin{proof} It is enough to prove the corresponding statement for the form $\Psi'$ where, in the definition of $\Psi$, we replace $\displaystyle \frac{\partial \psi}{\partial z}\wedge\overline{\frac{\partial \psi}{\partial z}}$ with $\displaystyle \frac{\partial \omega}{\partial z}\wedge\overline{\frac{\partial \omega}{\partial z}}$. Clearly, we can view $\Psi$ or $\Psi'$ as an element of $\mathcal{K}(S)\otimes _\Cee\bigwedge^{2h+2}H$. Consider the meromorphic sections 
$$\frac{\partial \omega}{\partial z}, \frac{\partial \omega}{\partial t_1}, \dots, \frac{\partial \omega}{\partial t_h}$$
of $F^2\cap(e_{h+1})^\perp$. Over the field of meromorphic functions on $S^*$, the span of 
$$\omega,  \frac{\partial \omega}{\partial z}, \frac{\partial \omega}{\partial t_1}, \dots, \frac{\partial \omega}{\partial t_h}$$
is the same as the span of 
$$\tilde\omega, \frac{\partial \tilde\omega}{\partial t_1}, \dots , \frac{\partial \tilde\omega}{\partial t_h}, \frac{\partial \tilde\omega}{\partial q}.$$
Hence $\displaystyle  \frac{\partial \omega}{\partial z}, \frac{\partial \omega}{\partial t_1}, \dots, \frac{\partial \omega}{\partial t_h}$ are linearly independent over the field of meromorphic functions on $S^*$ and hence on $S$. Since $F^2\cap (e_{h+1})^\perp \to \overline{\mathcal{H}}_\#$ is injective and of maximal rank, the above sections remain linearly independent when viewed as meromorphic sections of $F^2_\#$.    

Let $\sigma_0, \dots, \sigma_h$ be a basis of holomorphic sections for the holomorphic bundle $F^2_\#$. Then, by Lemma~\ref{Fsharp},  there exists a nonzero real analytic function $A$ such that
$$\sigma_0\wedge \bar{\sigma}_0\wedge \cdots \wedge \sigma_h\wedge \bar{\sigma}_h = Ae_0\wedge f_0\wedge \cdots \wedge e_h\wedge f_h.$$
There exists an $(h+1)\times (h+1)$ matrix $G$ whose entries are meromorphic functions on $S$ expressing the images of the meromorphic sections $\displaystyle  \frac{\partial \omega}{\partial z}, \frac{\partial \omega}{\partial t_1}, \dots, \frac{\partial \omega}{\partial t_h}$ in $F^2_\#$ as linear combinations of  $\sigma_0, \dots, \sigma_h$. Furthermore,  $\det G \neq 0$, because $\displaystyle  \frac{\partial \omega}{\partial z}, \frac{\partial \omega}{\partial t_1}, \dots, \frac{\partial \omega}{\partial t_h}$ are linearly independent over the field of meromorphic functions on $S$.  Then, working in $H_\#$ and the associated $C^\infty$ bundle (i.e.\ mod $e_{h+1}$), 
\begin{align*}
\Psi \bmod e_{h+1} =\Psi' \bmod e_{h+1}  &=
 \pm |\det G|^2\sigma_0\wedge \bar{\sigma}_0\wedge \cdots \wedge \sigma_h\wedge \bar{\sigma}_h   \\
&= \pm |\det G|^2Ae_0\wedge f_0\wedge \cdots \wedge e_h\wedge f_h.
\end{align*}
This says that, for some nonzero real meromorphic function $M_1$,
$$
\Psi =  M_1e_0\wedge f_0\wedge \cdots \wedge e_h\wedge f_h 
$$
mod $e_{h+1}$, and thus completes the proof of the lemma. 
\end{proof}

To finish the proof of Theorem~\ref{maincalc}, our goal is to calculate 
$$\Xi \wedge \Xi' = \omega\wedge \bar\omega \wedge \frac{\partial \omega}{\partial z} \wedge \overline{\frac{\partial \omega}{\partial z}} \wedge  \frac{\partial \omega}{\partial t_1} \wedge \overline{\frac{\partial \omega}{\partial t_1}} \wedge  \cdots \wedge\frac{\partial \omega}{\partial t_h} \wedge \overline{\frac{\partial \omega}{\partial t_h}}$$
which is the wedge product of $\displaystyle \Xi= \omega\wedge \bar\omega \wedge \frac{\partial \omega}{\partial z} \wedge \overline{\frac{\partial \omega}{\partial z}}$ with 
$$\Xi'= \frac{\partial \omega}{\partial t_1} \wedge \overline{\frac{\partial \omega}{\partial t_1}} \wedge  \cdots \wedge\frac{\partial \omega}{\partial t_h} \wedge \overline{\frac{\partial \omega}{\partial t_h}}.$$
In our previous notation,
$$\Psi = \frac{\partial \psi}{\partial z} \wedge \overline{\frac{\partial \psi}{\partial z}} \wedge \Xi'.$$
Since $\Xi'$  does not involve $f_{h+1}$,  any terms of $\Xi$ which do not involve $f_{h+1}$ will drop out of $\Xi\wedge \Xi'$.
By the above lemma and Equation $(*)$, $\Xi\wedge \Xi'$ is of the form 
\begin{align*}
\Xi\wedge \Xi'&= (z-\bar{z})\Psi \wedge  e_{h+1}\wedge f_{h+1} + \Phi \wedge \Xi'\\
&=(z-\bar{z})M_1e_0\wedge f_0\wedge \cdots \wedge e_{h+1}\wedge f_{h+1} + \Phi \wedge \Xi',
\end{align*}
with $M_1\neq 0$.
Since  the coefficients of $\Phi$, $\Xi'$  are real meromorphic functions, $\Phi \wedge \Xi'\in   \bigwedge^{2h+4}H\otimes_\Cee\mathcal{K}(S)$  and we can write
$$\Phi \wedge \Xi' = M_2e_0\wedge f_0\wedge \cdots \wedge e_{h+1}\wedge f_{h+1}$$
for some real meromorphic function $M_2$. Thus $\Xi\wedge \Xi'$ is as claimed. 
\end{proof}

\subsection{The main theorem} We can now prove the main theorem of this paper:

\begin{theorem} 
Let $X$ be a compact complex manifold of dimension $3$ with $K_X\cong \scrO_X$ for which the $\partial\bar{\partial}$-lemma holds, and such that that $H^i(X; \scrO_{X}) = 0$ for $i=1,2$ and  $H^0(X; \Omega_{X}^j) = 0$ for $j=1,2$.  Suppose that  $C_1, \dots, C_r$ are disjoint smooth rational curves in $X$ such that  $N_{C_i/X} \cong \scrO_{\Pee^1}(-1) \oplus \scrO_{\Pee^1}(-1)$. Assume that the classes  $[C_i]$ of the $C_i$ satisfy a linear relation in $H^4(X; \Cee)$ of the form
$$\sum_{i=1}^rm_i [C_i] =0,$$
where $m_i \neq 0$ for every $i$. Let $\overline{X}$ be the singular compact threefold obtained by contracting the $C_i$. Then there exist smoothings of $\overline{X}$ for which the $\partial\bar{\partial}$-lemma holds.
\end{theorem} 
\begin{proof}  If $s$ is the smallest positive integer such that there exists a subset of the $C_i$ whose classes are linearly dependent, then, possibly after reordering the $C_i$, we can assume that there exist $n_1, \dots, n_s\in \Q$, such that $n_i\neq 0$ for all $i$, 
$$\sum_{i=1}^sn_i [C_i] =0,$$   
and  the $[C_i]$ span a subspace of $H^4(X)$ of dimension $s-1$. Let $\overline{X}_1$ be the singular threefold obtained by contracting $C_1, \dots, C_s$ and let $X_1$ be a general small smoothing  of $\overline{X}_1$. By Corollary~\ref{firstreduction}, the $\partial\bar{\partial}$-lemma holds for $X_1$. In particular, we have proved the corollary in case $r=s$, and hence in case $r=1$. Now assume the result by induction for all positive integers less than $r$ and suppose that $s< r$. The curves $C_{s+1}, \dots, C_r$ deform to disjoint smooth rational curves $C_i'$ in $X_1$. Since $H^4(X_1) \cong H^4(X)/\sum_{j=1}^s\Cee\cdot [C_j]$,   $\sum_{i= s+1}^rm_i[C_i']=0$ in $H^4(X_1)$. Let $\overline{X}_1$ be the threefold obtained by contracting $C_{s+1}', \dots, C_r'$ in $X_1$. Then $\overline{X}_1$ is smoothable, and by induction the $\partial\bar{\partial}$-lemma holds for general small smoothings $X_2$ of $\overline{X}_1$. Such a smoothing will also be a general small smoothing of $\overline{X}$, completing the proof of the theorem.  
\end{proof}

\section{Concluding remarks}

First we recall the following standard definition:

\begin{definition} Let $V_1$ and $V_2$ be two compact complex manifolds. Then $V_1$ and $V_2$ are \textsl{deformation equivalent} if there exists a proper smooth morphism $\pi \colon\mathcal{V}\to S$, where $\mathcal{V}$ and $S$ are connected analytic spaces, and two points $s_1, s_2\in S$, such that $\pi^{-1}(s_i) \cong V_i$, $i=1,2$.
\end{definition}

I am grateful to S.-T.\ Yau for calling my attention to the following question: is every compact complex manifold for which the $\partial\bar{\partial}$-lemma holds deformation equivalent to a compact complex manifold bimeromorphic to a K\"ahler manifold (also called \textsl{of class $\mathcal{C}$})?   The answer to this question is no:

\begin{proposition} A Clemens manifold is not deformation equivalent to a compact complex manifold bimeromorphic to a K\"ahler manifold.
\end{proposition}
\begin{proof} Since the condition that $b_2=0$ is preserved under deformation equivalence, it suffices to show that a compact complex threefold $V$ with $b_2=0$ is not bimeromorphic to a K\"ahler manifold. Assume the contrary, that $V$ is bimeromorphic to a K\"ahler manifold $V'$. In fact, we may assume that there is a surjective degree one morphism $f\colon V'\to V$.  By \cite[(5.3)]{DeligneLefschetz} (cf.\ also \cite[(5.22)]{DGMS}), the $\partial\bar{\partial}$-lemma holds for $V$. As $b_2(V) =0$,   $h^2(V;\scrO_V) = h^0(V;\Omega^2_V) =0$. Since $h^0(V;\Omega^2_V)$ is a birational invariant, $h^0(V';\Omega^2_{V'})=0$ as well. Then $H^2(V';\Cee) = H^{1,1}(V')$, there exists a Hodge metric on $V'$ and so $V'$ is  projective. Since $f\colon V'\to V$ is birational,  there exists a hypersurface $E\subseteq V'$ such that $f(E)$ has codimension at least two and $f$ induces an isomorphism $V'-E \cong V-f(E)$. Choose an irreducible  curve $C$ on $V'$ not contained in $E$ and an irreducible   very ample divisor $H$ on $V'$ such that $H\cap C$ is finite and disjoint from $E$. Then $f(H)$ is a hypersurface in $V$ and it meets $f(C)$ at a finite and nonempty set of points. Then $[f(H)] \cup [f(C)]>0$, so that $[f(H)]$ is a nonzero element of $H^2(V;\Q)$. This contradicts the assumption  that $b_2(V) = 0$.
\end{proof} 

\begin{remark} Let $\pi\colon \mathcal{X} \to \Delta$ be a degeneration of compact complex manifolds $X_t, t\neq 0$, to a singular $X_0$. Under very general hypotheses, the arguments of Theorem~\ref{MHS2} will show that, for $t\in \Delta^*$ small, the Hodge-de Rham spectral sequence for $X_t$ degenerates at $E_1$. For example, if all components of $X_0$  are bimeromorphic to K\"ahler manifolds, or   if $X_0$ has normal crossings and all  $k$-fold intersections $X_0^{[k]}$ satisfy the $\partial\bar{\partial}$-lemma, then the Hodge-de Rham spectral sequence for $X_t$ degenerates at $E_1$ for $t$ small and $\neq 0$. On the other hand, it is easy to find examples for which the $\partial\bar{\partial}$-lemma does not hold for $X_t$, $t\neq 0$. For example, let $X_0$ be the singular surface which is obtained by gluing the negative section $\sigma_0$ of the rational ruled surface $\mathbb{F}_n$ to a disjoint section $\sigma$ by some choice of isomorphism. Note that $\sigma_0^2 = -n$ and $\sigma^2=n$, so that $X_0$ is $d$-semistable in the sense of \cite{Friedman1983}. Kodaira has shown \cite{Kodaira} that, if $n\neq 0$, then there is a degeneration $\pi \colon \mathcal{X} \to \Delta$, such that $\pi^{-1}(0) \cong X_0$ and, for $t\neq 0$, $X_t =\pi^{-1}(t)$ is a  Hopf surface. Then the Hodge-de Rham spectral sequence for $X_t$ degenerates at $E_1$, but the $\partial\bar{\partial}$-lemma does not hold for any compact complex surface deformation equivalent to $X_t$.
\end{remark}

\end{document}